\documentclass{amsart}

\usepackage[hidelinks]{hyperref}
\usepackage{fontspec}
\usepackage{unicode-math}
\usepackage{minted}
\usepackage{titlesec}

\titleformat{\section}
  {\Large\bfseries}
  {\thesection}
  {1em}
  {}

\setminted{
  fontsize=\footnotesize,
  breaklines,
  autogobble,
}

\newtheorem{lemma}{Lemma}
\newtheorem{theorem}[lemma]{Theorem}
\newtheorem{corollary}[lemma]{Corollary}

\begin{document}

	\title[Gram higher derivatives]{Higher derivatives of the Gram function}
\author[M. Dundulis]{Marius Dundulis}	
\address[Marius Dundulis]{Euromonitor International, Jogailos g. 4, LT-01116 Vilnius, Lithuania }
\email{marius.dundulis@euromonitor.com}

\author[R. Garunkštis]{Ramūnas Garunkštis}
        \address[Ramūnas Garunkštis]{Institute of Mathematics\\
	   Faculty of Mathematics and Informatics\\
	   Vilnius University\\
	   Naugarduko 24, 03225 Vilnius, Lithuania}
	   \email{ramunas.garunkstis@mif.vu.lt}
	   \urladdr{https://mif.vu.lt/~garunkstis/}
    
	\author[A. Laurinčikas]{Antanas Laurinčikas}
        \address[Antanas Laurinčikas]{Institute of Mathematics\\
        Faculty of Mathematics and Informatics\\
        Vilnius University\\
        Naugarduko 24, 03225 Vilnius, Lithuania}
        \email{antanas.laurincikas@mif.vu.lt}
 
	\author[R. Šimėnas]{Raivydas Šimėnas}
        \address[Raivydas Šimėnas]{Institute of Mathematics\\
		  Faculty of Mathematics and Informatics\\
		Vilnius University\\
		Naugarduko 24, 03225 Vilnius\\
		Lithuania}
	\email{raivydas.simenas@mif.vu.lt}
	\urladdr{https://web.vu.lt/mif/r.simenas/}
	
	\date{August 27, 2026}

    \subjclass[2020]{11M99}

    \keywords{Gram function}
	
	\begin{abstract}
		We study the nonlinear phenomenon of the asymptotics of higher-order derivatives of the Gram function $t_u$. We also prove the uniform distribution and continuous uniform distribution modulo one of $t^n_u$ for $n = 1, 2, \dots$. We use AI to formalize our proofs in Lean.
	\end{abstract}

    \maketitle

 \section{Introduction and main results}
	
	The Danish mathematician Gram introduced the concept of Gram points. They are points on the critical line where the Riemann-Siegel theta function $\vartheta (t)$ is a multiple of $\pi$. Using a heuristic argument, Gram \cite{Gram1903Note} reasoned that there is precisely one zero of the Riemann zeta function between two neighboring Gram points. Later on, Hutchinson \cite{Hutchinson1925RootsRiemann} found counterexamples to this conjecture. Nevertheless, the Gram points, which behave nonlinearly, remain an active area of research.

The Riemann-Siegel theta function $\vartheta(t)$ gives us an increment of the continuous branch of the argument of $\pi^{-s/2} \Gamma(s/2)$ as $s$ goes along the straight line segment of the complex plane joining $s=1/2$ and $s = 1/2 + it$; note that $\vartheta(t)$ is increasing for $t\ge 7$ (Ivi\'c \cite[Section 1.3]{Ivic2013} and Korolev \cite[p. 3]{korolev2016gram}). The function $\vartheta$ appears in the symmetric functional equation involving the Riemann zeta function $\zeta$
	\begin{displaymath}
		e^{i\vartheta(t)} \zeta(1/2 + it) = e^{-i\vartheta(t)} \zeta(1/2 - it).
	\end{displaymath}

The function $\vartheta$ is related to the function $S$.
	The function $S(t)$ is defined as
	\begin{displaymath}
		S(t) = \frac{1}{\pi} \arg \zeta(1/2 + it).
	\end{displaymath}
	Here, the argument is taken along the continuous branch of $\arg \zeta(s)$ as we move from $s = 2$ to $s = 2 + it$ and $s = 1/2 + it$. If $s$ is a zero of $\zeta$, then we set
	\begin{displaymath}
		S(t) = \lim_{\delta \to +0} \frac{1}{2}(S(t + \delta) + S(t - \delta)).
	\end{displaymath}

The Riemann-von Mangoldt formula (see, e.g. Edwards \cite[Section 6.6]{Edwards2001}, Ivi\'c \cite[formula (1.45)]{Ivic2013}, Karatsuba and Korolev \cite{karatsuba2005argument}) leads to the following relation between the functions $\vartheta$ and $S$
	\begin{equation}
	\label{eq:thetasrelation}
	    N(t) = \frac{1}{\pi} \vartheta(t) + 1 + S(t).
	\end{equation}
	Here $N(t)$ is the number of $\zeta$ zeros in the rectangle $0 < \Im s \leq t$, $0 \leq \Re s \leq 1$, with the exception that at points of discontinuity, which coincide with the ordinates of the zeros $\gamma_n$, it is defined as the limit
	\begin{displaymath}
	    N(t) = \frac{1}{2} \lim_{h \to 0} (N(t + h) + N(t - h)).
	\end{displaymath}
	The function $S$ is related to the function $N$ by the Riemann-von Mangoldt formula.
	\begin{displaymath}
		N(t) = \frac{t}{2\pi} \log \frac{t}{2\pi} - \frac{t}{2\pi} + \frac{7}{8} + S (t) + \frac{1}{\pi} \delta (t)
    \end{displaymath}
	(see Karatsuba and Korolev, \cite[Theorem 1]{karatsuba2005argument}).
	
	From the above equation, the function $S$ can be expressed as
	\begin{equation}\label{eq:Sfull}
		S(t) = -\frac{t}{2\pi}\log \frac{t}{2\pi} + \frac{t}{2\pi} - \frac{7}{8} - \frac{1}{\pi} \delta(t) + N(\gamma + 0)
	\end{equation}
	with $\gamma < t < \gamma'$, where $\gamma$ and $\gamma'$ are neighboring ordinates of the zeros of $\zeta$ (see Karatsuba and Korolev \cite[Proof of Theorem 2]{karatsuba2005argument}). The function $\delta$ is
	\begin{equation}\label{eq:deltadefinition}
		\delta(t) = \frac{t}{4} \log \left( 1 + \frac{1}{4t^2} \right) + \frac{1}{4}\arctan \frac{1}{2t} - \frac{t}{2} \int_0^\infty \frac{\rho(u) \, du}{(u+1/4)^2 + (t/2)^2}
	\end{equation}
	and $\rho(u) = 1/2 - \{ u \}$ (see Karatsuba and Korolev \cite[Theorem 1]{karatsuba2005argument}).

	The first derivative of $S$ is
    \begin{equation} \label{eq:sfirstder}
	    S'(t) = -\frac{1}{2\pi}\log \frac{t}{2\pi} + O(t^{-2}), \ t \to +\infty
	\end{equation}
	except for its points of discontinuity (see Karatsuba and Korolev \cite[Theorem 2]{karatsuba2005argument}).

    We give our theorem about the $n$-th derivative of the $S$ function. Note that in the theorem below and elsewhere, the implied constants depend on $n$, where $n$ denotes the order of the derivative of our functions $S$, $\vartheta$, and $t_u$.
	\begin{theorem} \label{theorem:sderivative}
		The $n$-th derivative of the $S$ function when $n \geq 2$ is
		\begin{equation}\label{eq:snthder}
			S^{(n)}(t) = (-1)^{n-1} \frac{(n-2)!}{2\pi}t^{1-n} + O(t^{-n-1}), \ t \to +\infty.
		\end{equation}
  except for its discontinuity points.
	\end{theorem}

    From the above theorem, equation \eqref{eq:thetasrelation}, equation \eqref{eq:sfirstder}, and the fact that $\vartheta$ is differentiable infinitely many times, we get
    \begin{corollary}\label{corollary:thetanderivative}
	Suppose $n \geq 2$. Then, the $n$-th derivative of the Riemann-Siegel theta function is
	\begin{equation}\label{eq:thetanthder}
	    \vartheta^{(n)}(t) = (-1)^{n} \frac{(n-2)!}{2}t^{1-n} + O(t^{-n-1}), \ t \to +\infty.
	\end{equation}
	\end{corollary}

It is known that for $t \geq 7$, the Riemann-Siegel theta function $\vartheta(t)$ is monotonically increasing. Thus for $u \geq \vartheta(7)/\pi + 1$, we can implicitly define the Gram function $t_u$ by
\begin{equation}\label{eq:gramimplicit}
		\vartheta(t_u) = (u-1)\pi,
	\end{equation}
with $t_u \geq 7$. Note that for positive integers $k$, $t_k$ correspond to Gram points (see Edwards \cite[pp. 125--126]{Edwards2001}). Elsewhere in the literature, the reader might encounter the definition $\vartheta(g_m) = m \pi$, where $g_m$ denotes the $m$-th Gram point. This numeration relates to ours by $g_{k - 1} = t_k$.

 The asymptotic expression for the Gram function is (see Lavrik \cite[Lemma 2]{lavrik1994titchmarsh})
	\begin{equation}\label{eq:gramestimate}
		t_u = \frac{2\pi u}{\log u} \left( 1 + (1 + o(1)) \frac{\log \log u}{\log u} \right), \ u \to +\infty.
\end{equation}

 Korolev \cite[Lemma 1.1]{korolev2016gram} gave the asymptotics for the first two derivatives of the Gram function. The first derivative of the Gram function is
 \begin{equation}
		\frac{d t_u}{du} = \frac{2\pi}{\log u}\left(1 + (1 + o(1))\frac{\log \log u}{\log u} \right), \label{eq:gramfirstder}
	\end{equation}
	as $u \to +\infty$. 

We calculate the $n$-th derivative of the Gram function $t_u$, which we denote by $t_u^{(n)}$.
    \begin{theorem}
        \label{theorem:gramnthderivative}
        The $n$-th derivative of the Gram function is asymptotically
        $$
            t_u^{(n)} = \frac{(-1)^{n + 1} 2 \pi (n - 2)!}{u^{n - 1} \log^2 u} \left( 1 + (2 + o(1)) \frac{\log \log u}{\log u}\right), \ u \to +\infty,
        $$
        for $n = 2, \ldots$. 
    \end{theorem}

    Here, we corrected the constant in \cite[Lemma 1.1]{korolev2016gram} for $n = 2$.

    The proof of the following theorem of the uniform distribution modulo one of the sequence $\{t^n_k\}_{k = 1}^\infty$ is analogous to Pańkowski \cite[Proof of Theorem 1]{Pankowski2018JointL-functions}.
    \begin{theorem}
        \label{thm:gramunifdist}
        The sequence $\{t_k^n\}_{k = 1}^\infty$ is uniformly distributed modulo one for $n = 1, 2, 3, \ldots$.
    \end{theorem}

    We have the following corollary about the continuous uniform distribution of the $n$-th power of the Gram function. Laurinčikas proved this result \cite[Lemma 5.7]{Laurincikas2024Universality} by a different method. In our case, it follows by applying the derivative criterion for uniform distribution modulo one for the sequences $\{t_{k + a}^n \}_{k = 1}^\infty$ and $\{ t_{ak}^n \}_{k = 1}^\infty$.
    \begin{corollary}\label{prop:gramcontunifdist}
        The function $t_u^n$, $n = 1, 2, 3, \ldots$, is continuously uniformly distributed modulo one.
    \end{corollary}

    The function $t_u$ can be applied in the universality theory of zeta functions. As an example, we will consider the Riemann zeta function. Recall that the universality of $\zeta(s)$ in the Voronin sense means that a broad class of analytic functions can be approximated by shifts $\zeta (s + iu)$, $u \in \mathbb{R}$. This approximation property of $\zeta(s)$ was discovered by Voronin \cite{Voronin1975TheoremZeta-function} and has various theoretical and practical applications. For results, see Matsumoto \cite{Matsumoto15}.

    Now we discuss the discrete universality for the Riemann zeta function using its shifts by the degrees of the Gram points $\zeta(s + it^n_k)$. Korolev and Laurinčikas have studied this problem extensively. In \cite{KorolevLaurincikas2019Gram}, they proved it for the degree $1$. They proved the result for short intervals in \cite{KorolevLaurincikas2020Gram}. The continuous version of joint universality was proved in \cite{KorolevLaurincikas2023Gram}. 
    
    By $\#A$ we denote the cardinality of the set $A$. In addition, $\zeta_{\{p : p \leq y\}}(s):= \prod_{p \leq y}(1 - 1/p^s)^{-1}$ is a truncated Euler product for the Riemann zeta function, where $p$ is prime. We assume about the second moment of the Riemann zeta function shifted by the Gram function: For all $\epsilon > 0$, we conjecture
    \begin{equation}\label{eq:assumption}
        \lim_{y \to \infty} \limsup_{N \to \infty} \sup_{s \in K} \frac{1}{N} \sum_{k = 2}^N | \zeta(s + it_k^n) - \zeta_{\{p : p \leq y\}}(s + it_k^n)|^2 < \epsilon,
    \end{equation}
     where $K$ is any compact set with connected complement in the right part of the critical strip. Pańkowski \cite[p. 192, Proposition 1]{Pankowski2018JointL-functions} incorrectly thought that this follows from Gallagher's lemma applied to Dirichlet $L$-functions with shifts $\alpha t^a(\log t)^b$ for $\alpha \neq 0$, $a > 0$, and $b$ real numbers. Sourmelidis \cite{Sourmelidis2022Weyl} partially remedied this error. The proof of universality follows from our assumption \eqref{eq:assumption} and Theorem \ref{thm:gramunifdist} using the method by Pańkowski \cite[pp. 192--194, proof of Theorem 2]{Pankowski2018JointL-functions}. We formulate our statement as follows: \emph{Suppose $K \subset \{ s \in \mathbb{C}: 1/2 < \sigma < 1 \}$ is a compact set with connected complement, $f$ is a continuous non-vanishing function on $K$, analytic inside of $K$. Then, for any $\epsilon > 0$,}
        $$
        \liminf_{N \to \infty} \frac{1}{N} \# \left\{ 2 \leq k \leq N : \max_{s \in K} |\zeta(s + it^n_k) - f(s)| < \epsilon \right\} > 0.
        $$

    In the last section of the article, we discuss the proof of our results using Lean. In doing so, we employed the AI models Claude Opus 4.7, Opus 4.8, and Fable 5. Our axiomatization is the standard one for classical mathematics.

	\section{Proofs}
	
	\begin{proof}[Proof of Theorem \ref{theorem:sderivative}]
		Our goal is to compute the $n$-th derivative of $S(t)$. Recall the formula \eqref{eq:Sfull}. We split the computation into two parts.
        
        First, let us compute
		\begin{displaymath}
			\left(-\frac{t}{2\pi} \log \frac{t}{2\pi} + \frac{t}{2\pi} - \frac{7}{8} \right)^{(n)} = \frac{d^{n - 2}}{dt^{n - 2}}\left(- \frac{1}{2\pi t} \right) = (-1)^{n - 1} \frac{(n-2)!}{2\pi t^{n-1}}.
		\end{displaymath}
		It remains to estimate the error term arising from differentiating the function $\delta$. The terms
		\begin{equation}
                \label{eq:firsttwoterms}
			\frac{t}{4}\log \left(1 + \frac{1}{4t^2} \right) + \frac{1}{4} \arctan \frac{1}{2t}
		\end{equation}
		can be expanded into
            $$
            \frac{3}{16t} + \sum_{k = 3}^\infty \frac{a_k}{t^k}
            $$
        with real coefficients $a_n$. This Laurent series is differentiable since it converges for $|t| > 1/2$ and is analytic. Differentiating the above expression $n$ times term-by-term yields
         \begin{equation}\label{eq:deltafirstterms}
         (-1)^n \frac{3n!}{16 t^{n + 1}} + O(t^{-n - 3}), \ t \to +\infty.
         \end{equation}
For our purposes, the accuracy $O(t^{-n - 1})$ is sufficient. 
				
		Continuing our computation of the $n$-th derivative of $\delta(t)$, let us denote
		\begin{align*}
			j(t) & := \int_0^\infty \frac{\rho(u) \, du}{(u + 1/4)^2 + (t/2)^2},\\
            \sigma(u) & := \int_0^u \rho(z) \, dz.
		\end{align*}
		We compute the $n$-th derivative of $j$. We follow with integration by parts
		\begin{displaymath}
			j(t) = \int_0^\infty \frac{d\sigma(u)}{(u + 1/4)^2 + (t/2)^2} = 2 \int_0^\infty \frac{\sigma(u)(u+1/4) \, du}{((u + 1/4)^2 + (t/2)^2)^2}.
		\end{displaymath}
		Note that $0 \leq \sigma(u) \leq 1/8$.
		
		We need to differentiate $j(t)$ with respect to $t$ $n$ times, which amounts to differentiating under the integral sign. We use the Leibniz rule 
        $$
            \frac{d}{d t} \int_a^\infty f(u, t) \, du = \int_a^\infty \frac{\partial}{\partial t} f(u, t) \, du.
        $$
        The following conditions must be satisfied: $f(u, t)$ and $\partial f/\partial t$ be continuous for $u \geq a$, and there exists an integrable $g(u)$ such that $|\partial f(u, t)/\partial t)| \leq g(u)$ (see Rudin \cite[pp. 180--182]{Rudin1976Analysis}). For integrating the $k$-th derivative, we need an integrable majorant uniform in $t$ on $[T, 
    \infty)$. This condition is satisfied by
        $$
            \left| \frac{\partial^k}{\partial t^k} \left( \frac{1}{((u + 1/4)^2 + (t/2)^2)^2}\right) \right| \ll_k \frac{1}{((u + 1/4)^2 + 1)^2}, \ t \geq 1
        $$
        and the bound $0 \leq \sigma \leq 1/8$.
        
		Applying the Leibniz rule yields
        \begin{align*}
			 \frac{d^n}{dt^n} & \int_0^\infty \frac{(u + 1/4) \sigma(u)\, du}{((u+1/4)^2 + (t/2)^2)^2} 
             = \int_0^\infty \frac{\partial^n}{\partial t^n} \frac{(u + 1/4) \sigma(u) \, du}{((u + 1/4)^2 + (t/2)^2)^2} \\
            & \ll \int_0^\infty\sum_{r=\lceil n/2 \rceil}^n \frac{t^{2r-n} (4u + 1) \sigma(u)\, du}{((4u + 1)^2 + 4t^2)^{r + 2}}, \ t \to +\infty.
		\end{align*}

        In evaluating the above integral, we use the bound $\sigma(u) \le 1/8$ and the substitution $v = (4u+1)^2 + 4t^2$, so that $dv = 8(4u+1) \,  du$. We get

\begin{align*}
& \int_0^\infty \frac{(4u+1)\sigma(u)du}{((4u+1)^2+4t^2)^{r+2}}
\le \frac{1}{8}\int_0^\infty \frac{(4u+1)du}{((4u+1)^2+4t^2)^{r+2}} \\ &
= \frac{1}{64}\int_{1+4t^2}^{\infty} \frac{dv}{v^{r+2}}
= \frac{1}{64(r+1)(1+4t^2)^{r+1}}.
\end{align*}

Multiplying by $t^{2r-n}$ and using $1+4t^2 \ge 4t^2$:

\begin{align*}
& \frac{t^{2r-n}}{64(r+1)(1+4t^2)^{r+1}} \le \frac{t^{2r-n}}{64(r+1)4^{r+1}t^{2r+2}} \\ & = \frac{t^{-n-2}}{64(r+1)4^{r+1}} .
\end{align*}

Summing over $\lceil n/2\rceil \le r \le n$  yields:

$$
j^{(n)}(t) \ll_n t^{-n-2}, \ t \to +\infty.
$$
         
         We need to find the $n$-th derivative of the last term of the function $\delta(t)$ (see \eqref{eq:deltadefinition})
         $$
          - \frac{t}{2} \int_0^\infty \frac{\rho(u) \, du}{(u+1/4)^2 + (t/2)^2} = -\frac{t}{2} j(t).
         $$
          We apply the Leibniz rule to differentiate products. If $u$ and $v$ are $n$ times differentiable functions of $x$, then (see Gradshteyn and Ryzhik \cite[Formula 0.42]{gradshteyn2014table})
		 \begin{displaymath}
		 	\frac{d^n(uv)}{dx^n} = u \frac{d^nv}{dx^n} + \binom{n}{1} \frac{du}{dx}\frac{dv^{n-1}}{dx^{n-1}} + \binom{n}{2}\frac{d^2 u}{dx^2}\frac{d^{n-2}v}{dx^{n-2}} + \ldots + \frac{d^n u}{dx^n}v.
		 \end{displaymath}
	   We get
		 \begin{displaymath}
		 	-\frac{1}{2} \left(t j^{(n)}(t) + \binom{n}{1} j^{(n-1)}(t) \right) = O(t^{-n-1}).
		 \end{displaymath}
		 Collecting this and \eqref{eq:deltafirstterms} completes the proof of Theorem \ref{theorem:sderivative}.
	\end{proof}

    \begin{proof}[Proof of Theorem \ref{theorem:gramnthderivative}]
        We prove this by induction. Let us make the inductive hypothesis
        \begin{equation} \label{eq:inductivestep} 
            t_u^{(k)} = \frac{(-1)^{k + 1} 2 \pi (k - 2)!}{u^{k - 1} \log^2 u} \left( 1 + (2 + o(1)) \frac{\log \log u}{\log u}\right), \ u \to +\infty.
        \end{equation}
        First, we prove that the theorem statement holds for $k = 2$. Differentiating the equation \eqref{eq:gramimplicit} gives
     $$
     \vartheta''(t_u) t^{\prime 2}_u + \vartheta'(t_u) t''_u = 0.
     $$
     We calculate the term $\vartheta''(t_u) t^{\prime 2}_u$ using Corollary \ref{corollary:thetanderivative}, equations \eqref{eq:gramfirstder} and \eqref{eq:gramestimate}
     \begin{align*}
         \vartheta''(t_u) t^{\prime 2}_u = & \left(\frac{1}{2 t_u} + O(t_u^{-3}) \right) \left(\frac{2\pi}{\log u}\left(1 + (1 + o(1))\frac{\log \log u}{\log u} \right) \right)^2 \\ = &\frac{2 \pi}{2 u \log u}\left( 1 + (1 + o(1)) \frac{\log \log u}{\log u}\right), \ u \to +\infty.
     \end{align*}

     Differentiating \eqref{eq:gramimplicit} and using the expression \eqref{eq:gramfirstder} yields
     $$
     \frac{1}{\vartheta'(t_u)} = \frac{2}{\log u} \left(1 + (1 + o(1))\frac{\log \log u}{\log u} \right), \ u \to +\infty.
     $$
     Combining the above two results gives
     $$
     t''_u = - \frac{2\pi}{u \log^2 u}\left( 1 + (2 + o(1)) \frac{\log \log u}{\log u}\right), \ u \to +\infty.
     $$

     Suppose our hypothesis \eqref{eq:inductivestep} holds for $2 \leq k \leq n - 1$. We show it holds for $k = n$. By Faà di Bruno's formula (see Gradshteyn and Ryzhik \cite[formula (0.430)]{gradshteyn2014table}), the $n$-th derivative of a composite function $h(x) = g(f(x))$ is
\begin{align}\label{bruno}
\frac{d^n}{dx^n} g(f(x)) = & \sum \frac{n!}{m_1! 1!^{m_1} m_2! 2!^{m_2} \cdots m_n! n!^{m_n}} g^{(m_1 + m_2 + \cdots +m_n)}(f(x)) \\ & \times \prod_{j = 1}^n (f^{(j)}(x))^{m_j}. \nonumber
\end{align}
Here, the summation is taken over all partitions
\begin{equation} \label{eq:partition}
 m_1 + 2 m_2 + \cdots + n m_n = n.
\end{equation}
We will need the shorthand notation $\sum m = m_1 + m_2 + \cdots + m_n$ for our purposes.
        
We plug $g(s) = \vartheta(s)$, $f(u) = t_u$ into Faà di Bruno's formula. Then, the relation \eqref{eq:gramimplicit} yields that the $d^n \vartheta(t_u)/d u^n$ equals $0$ and the formula \eqref{bruno} becomes
\begin{align}\label{brunoforgram}
   0 =  \vartheta'(t_u) t_u^{(n)} + \cdots + \vartheta^{(n)}(t_u)(t_u')^n.
\end{align}
The first member in the above equation corresponds to the case $m_n = 1$ and $m_i = 0$ for $i = 1, \ldots, n - 1$. The last member - to the case $m_1 = n$ and $m_i = 0$ for $i = 2, \ldots, n$.
        
In the summation of Faà di Bruno's formula, the first term in the above equation \eqref{brunoforgram} is the only one containing $t_u^{(n)}$, and the last one is asymptotically the largest of the remaining ones as $u \to \infty$. Let us prove this fact. We will show that $\vartheta^{(n)}(t_u)(t'_u)^n$ is asymptotically larger than the terms obtained using Faà di Bruno's formula, except possibly the first one. Hence, we can ignore the constants. Any term is uniquely determined by the values $m_1$, $m_2$, \ldots, $m_n$, which come from the partition \eqref{eq:partition}. For this calculation, we use the asymptotics for the derivatives of $\vartheta(t)$ (Corollary \ref{corollary:thetanderivative}), the asymptotics for $t'_u$ and $t_u$, and our inductive step \eqref{eq:inductivestep}. We skip the term $\vartheta'(t_u) t^{(n)}_u$. In all the remaining terms, we have $m_n = 0$. The arbitrary term in Faà di Bruno's formula, except for the partition $m_n = 1$, $m_i = 0$, $i = 1, \ldots, n - 1$, is
\begin{align}
\nonumber &\left| \vartheta^{(m_1 + m_2 + \cdots +m_{n - 1})}(t_u) \prod_{j = 1}^{n - 1} (t^{(j)}_u)^{m_j} \right| = \Theta \left( \frac{1}{t_u^{\sum m - 1}} \cdot \frac{1}{\log^{m_1} u}\prod_{j = 2}^{n - 1} \left( \frac{1}{u^{j - 1} \log^2 u} \right)^{m_j} \right)
        \\ \nonumber
            & =
\Theta \left(\frac{\log^{\sum m -1}u}{u^{\sum m - 1} \log^{m_1} u} \cdot \frac{1}{u^{\sum_{j = 2}^{n -1} j m_j - \sum m + m_1} \log^{2 \sum m - 2m_1}u}\right) \\ \label{eq:memberofsumfaadibruno}
            & = \Theta \left( \frac{1}{u^{n - 1} \log^{\sum m - m_1 + 1} u} \right).
\end{align}
We say that a function $f(x)$ is $\Theta(g(x))$ if there exist positive constants $c_1$, $c_2$ and a real number $x_0$ such that $c_1 g(x) \leq f(x) \leq c_2 g(x)$, $x \geq x_0$. In the last step above, we used the equation \eqref{eq:partition}. The last expression shows that among all the terms in Faà di Bruno's formula, except $\vartheta'(t_u)t_u^{(n)}$, the asymptotically largest one is $\vartheta^{(n)}(t_u)(t'_u)^n$, corresponding to the partition $m_1 = n$ and $m_i = 0$, $i = 2, \ldots n$. This, together with \eqref{brunoforgram}, leads to
\begin{align*}
\vartheta'(t_u) t_u^{(n)}= -\vartheta^{(n)}(t_u)(t_u')^n +o\left( \vartheta^{(n)}(t_u)(t_u')^n \right)\qquad (u\to\infty).
 \end{align*}
\begin{align*}
\vartheta'(t_u) t_u^{(n)}= -\vartheta^{(n)}(t_u)(t_u')^n +O\left( \frac{1}{u^{n - 1} \log^{2} u} \right)\qquad (u\to\infty),
\end{align*}
It follows
$$
t_u^{(n)} = \frac{-\vartheta^{(n)}(t_u)(t'_u)^n}{\vartheta'(t_u)} + O \left( \frac{1}{\vartheta'(t_u) u^{n-1} \log^2 u}\right).
$$
The dominating term of $\vartheta'(t_u)$ is asymptotically $\log u$, thus the error term in the above expression becomes
$$
    O\left( \frac{1}{u^{n-1} \log^3 u}\right).
$$
Thus Corollary \ref{corollary:thetanderivative} together with the asymptotics for $t'_u$ and $t_u$ leads to
$$
t_u^{(n)} = \frac{(-1)^{n + 1} 2 \pi (n - 2)!}{u^{n - 1} \log^2 u} \left( 1 + (2 + o(1)) \frac{\log \log u}{\log u}\right).
$$
        This completes the induction and yields the theorem.
    \end{proof}
	
	\begin{proof}[Proof of Theorem \ref{thm:gramunifdist}]
	   We have theorem \cite[Chapter 1, Theorem 3.5]{Kuipers2012UniformSequences} stating: if there exists $l \in \mathbb{N}$, the function $f(x)$ is defined for $x \geq 1$, the derivative $f^{(l)}(x)$ tends monotonically to $0$ as $x \to \infty$, and if $x|f^{(l)}(x)| \to \infty$, $x \to \infty$, then the sequence $\{ f(n) \}_{n \in \mathbb{N}}$ is uniformly distributed modulo one.

    Consider the $n$-th power of the Gram function $t^n_u$. By Faà di Bruno's formula, we have
    \begin{equation} \label{eq:tnfaadibruno}
    \frac{d^n t_u^n}{du^n} = \sum \frac{n!}{m_1!1!^{m_1}m_2!2!^{m_2} \cdots m_n!n!^{m_n}}\frac{d^{m_1 + m_2 + \cdots m_n}t_u^n}{d t_u^{m_1 + m_2 + \cdots m_n}} \prod_{j = 1}^n (t_u^{(j)})^{m_j},
    \end{equation}
    where we sum over all partitions
    \begin{equation}\label{eq:npartition}
        m_1 + 2 m_2 + \cdots + n m_n = n.
    \end{equation}
    Let us denote $\sum m := m_1 + m_2 + \cdots + m_n$. An arbitrary term on the right-hand side of the equation \eqref{eq:tnfaadibruno} depends on the value of the vector $(m_1, m_2, \ldots, m_n)$. Ignoring the constant, recalling the asymptotic of $t_u$ given in the formula \eqref{eq:gramestimate}, and the asymptotic of its derivative in Theorem \ref{theorem:gramnthderivative} yields the following asymptotic of an arbitrary term in the formula \eqref{eq:tnfaadibruno}
    \begin{align*}
    & \frac{u^{n - \sum m}}{\log^{n - \sum m} u} \cdot \frac{1}{\log^{m_1} u} \prod_{j = 2}^n \frac{1}{(u^{j-1} \log^2 u)^{m_j}} = \frac{u^{n - \sum m}}{\log^{n - \sum m} u} \cdot \frac{1}{u^{n - \sum m} \log^{-m_1 + 2 \sum m} u} \\ & = \frac{1}{\log^{n -m_1 + \sum m} u},
    \end{align*}
    where we used the formula \eqref{eq:npartition}. This value is maximized when we minimize $-m_1 + \sum m$, which happens when $m_1 = n $ and $m_2 = \cdots = m_n= 0$. Here $\sum m \leq n$ thus no term degenerates.

    To prove monotonicity, we calculate the $n + 1$-st derivative of the $n$-th power of the Gram function. Again, we use Faà di Bruno's formula. Summing over partitions, we have
    $$
    m_1' + 2m_2' + \cdots  + nm_n' + (n + 1)m_{n + 1}' = n + 1,
    $$
    where the integers $m_1'$, $m_2'$, \ldots, $m_{n + 1}'$ appear in our formula now. Let us denote $\sum m' := m_1' + m_2' + \cdots + m_{n + 1}'$. It follows
    \begin{align*}
        & \frac{u^{n - \sum m'}}{\log^{n - \sum m'} u} \cdot \frac{1}{\log^{m_1'} u} \prod_{j = 2}^{n + 1} \frac{1}{(u^{j-1} \log^2 u)^{m_j'}} = \frac{u^{n - \sum m'}}{\log^{n - \sum m'} u} \\ & \times \frac{1}{u^{n + 1- \sum m'} \log^{-m_1' + 2 \sum m'} u} \\ & = \frac{1}{u\log^{n -m_1' + \sum m'} u}
    \end{align*}
    Again, this value is maximized by setting $\sum m' - m'_1 = 1$. 
    
    All in all, we have
    $$
    \frac{d^n t_u^n}{d u^n} \sim  \frac{n!(2\pi)^n}{\log^n u}
    $$
    and
    $$
        \frac{d^{n + 1} t_u^n}{du^{n + 1}} \sim  -\frac{n n!(2\pi)^n}{u \log^{n + 1} u}.
    $$
     Based on the fact that the main terms of the $n + 1$-st derivative's asymptotic is non-zero of a fixed sign, there exists some value $u_0$ for which $d^{n + 1} t_u^n/du^{n + 1}$ does not change its sign provided $u > u_0$. Hence, the $n$-th derivative of $t_u^n$ will monotonically approach $0$ for $u > u_0$. In addition, $u |d^n t_u^n/du^n|$ will monotonically approach $+\infty$. Therefore, by the criterion stated at the beginning of the proof, the sequence $\{t^n_k\}_{k = 1}^\infty$ is uniformly distributed modulo one.
    \end{proof}

    \begin{proof}[Proof of Corollary \ref{prop:gramcontunifdist}]
        By Kuipers and Niederreiter \cite[Chapter 1, Theorem 9.6]{Kuipers2012UniformSequences}, the function $f(t)$, $t \geq 0$, is continuously uniformly distributed modulo one if 1) the sequence $f(k+t)$, $k = 0, 1, \ldots$, is uniformly distributed modulo one for almost all $t \in [0, 1)$ and 2) the sequence $f(kt)$, $k = 1, 2, \ldots$, is uniformly distributed modulo one for almost all $t \in [0, 1)$. If we take $f$ to be the $n$-th power of the Gram function, both of these facts follow from the fact that there exists real number $u_0$ such that $d^n t^n_u/du^n$ monotonically approaches zero for $u > u_0$, $u|d^n t_u^n/du^n|$ monotonically approaches $+\infty$ for $u > u_0$, and the criterion for the uniform distribution modulo $1$ of sequences, formulated at the beginning of the proof of Theorem \ref{thm:gramunifdist}.
    \end{proof}

    \section{Using AI and Lean}

    Lean has emerged as the definitive programming language for proving mathematical theorems. We use Lean v4.29.1 to formalize our results. To write the code, we employed Claude Opus 4.7, Opus 4.8, and Fable 5 LLMs. The overall amount of code is substantial, more than 10000 lines. It is available at \url{https://github.com/raivydassimenas/gram-derivatives.git}. The proofs of our theorems and corollaries are located in the folder \texttt{GramDerivatives}. Each theorem or corollary has its own file, along with additional files proving the necessary results from uniform distribution theory. The project builds cleanly, with no additional axioms apart from the standard mathematical ones and no \mintinline{lean}{sorry} statements.

    Here, for the reader's convenience, we provide the formalizations of our theorems and corollaries in Lean
    \begin{itemize}
        \item Theorem 1. While in the main part of our article we studied a specific function $S$, in Lean \mintinline{lean}{S} does not depend on the analytic nature of the piecewise constant function $N(t)$. The important fact is that $N(t)$ is constant except at a discrete set of jump points, with the function locally constant off that set. Thus, we prove the result not for one particular function but for the whole class of functions. The function class $S$ is defined as
        \begin{minted}{lean4}
noncomputable def S (F : StepFunction) (t : ℝ) : ℝ :=
  φ t - (1 / Real.pi) * δ t + F t
        \end{minted}
        with \mintinline{lean}{StepFunction} defined as
        \begin{minted}{lean4}
structure StepFunction where
  toFun : ℝ → ℝ
  jumpSet : Set ℝ
  jumpSet_discrete :
    ∀ x ∈ jumpSet, ∃ ε > 0, ∀ y ∈ jumpSet, |y - x| < ε → y = x
  locallyConstant_off :
    ∀ ⦃s : ℝ⦄, s ∉ jumpSet → toFun =ᶠ[nhds s] fun _ => toFun s
        \end{minted}
        The function \mintinline{lean}{phi} is
        \begin{minted}{lean4}
noncomputable def φ (t : ℝ) : ℝ :=
  -(t / (2 * Real.pi)) * Real.log (t / (2 * Real.pi))
  + t / (2 * Real.pi)
  - 7 / 8
        \end{minted}
        The function \mintinline{lean}{delta} is
        \begin{minted}{lean4}
noncomputable def δ (t : ℝ) : ℝ :=
  t / 4 * Real.log (1 + 1 / (4 * t ^ 2))
  + 1 / 4 * Real.arctan (1 / (2 * t))
  - t / 2 * ∫ u in Set.Ici (0 : ℝ),
        (1 / 2 - Int.fract u) / ((u + 1 / 4) ^ 2 + (t / 2) ^ 2)
        \end{minted}
        The theorem itself becomes
        \begin{minted}{lean4}
theorem theorem1 (F : StepFunction) (n : ℕ) (hn : 2 ≤ n) :
    IsO
      (fun t =>
        iteratedDeriv n (S F) t
        - ((-1 : ℝ) ^ (n - 1) * (n - 2).factorial
           * (1 / (2 * Real.pi)) * t ^ (1 - (n : ℝ))))
      (fun t => t ^ (-(n : ℝ) - 1))
      N∞₀[F.jumpSet]        
      \end{minted}
        The predicate \mintinline{lean}{IsO(f g)} holds if and only if $f = O(g)$. The function \mintinline{lean}{iteratedDeriv n f t} is $f^{(n)}(t)$. The notation \mintinline{lean}{N^\infty_0} essentially means that $f(x) = O(g(x)), x \to \infty$. 

        \item Corollary 2 is an immediate consequence of Theorem 1. We formalize it as
        \begin{minted}{lean4}
theorem corollary2 (n : ℕ) (hn : 2 ≤ n) :
    IsO
      (fun t =>
        iteratedDeriv n theta t
        - ((-1 : ℝ) ^ n * (n - 2).factorial
           * (1 / 2) * t ^ (1 - (n : ℝ))))
      (fun t => t ^ (-(n : ℝ) - 1))
      N∞ 
      \end{minted}
        Here \mintinline{lean}{theta} is defined as
        \begin{minted}{lean4}
noncomputable def theta (t : ℝ) : ℝ := δ t - Real.pi * φ t - Real.pi
        \end{minted}
        \item Theorem 3
        \begin{minted}{lean4}
theorem theorem3 (n : ℕ) (hn : 2 ≤ n) :
    Iso
      (fun u : ℝ =>
        iteratedDeriv n gram u
        - gramLeading n u
        - 2 * gramLeading n u * Real.log (Real.log u) / Real.log u)
      (fun u : ℝ => gramLeading n u * Real.log (Real.log u) / Real.log u)
      N∞
      \end{minted}
        For this theorem, we define the Gram function as the inverse of the\\Riemann-Siegel theta function $\vartheta$
        \begin{minted}{lean4}
noncomputable def gram (u : Real) : Real :=
  Function.invFunOn theta (Set.Ici 7) ((u - 1) * Real.pi)
        \end{minted}
        The function \mintinline{lean}{gramLeading} is
        \begin{minted}{lean4}
noncomputable def gramLeading (n : ℕ) (u : ℝ) : ℝ :=
  (-1 : ℝ) ^ (n + 1) * (2 * Real.pi) * (n - 2).factorial
    / (u ^ (n - 1) * Real.log u ^ 2)
        \end{minted}
        \item Theorem 4
        \begin{minted}{lean4}
theorem theorem4 (n : ℕ) (hn : 1 ≤ n) :
    Gram.UD.IsUDModOne (fun k : ℕ => gramPow n (k : ℝ))
    \end{minted}
    The function \mintinline{lean}{gramPow} is defined as
        \begin{minted}{lean4}
noncomputable def gramPow (n : ℕ) (u : ℝ) : ℝ := (gram u) ^ n
        \end{minted}
        Here, we use auxiliary results such as the van der Corput bound and the Fejer theorem from uniform distribution theory, which we prove in the files \texttt{Fejer.lean}, \texttt{UDModOne.lean}, and \texttt{VanDerCorput.lean}.

        \item Corollary 5
        \begin{minted}{lean4}
lemma corollary5 (n : ℕ) (hn : 1 ≤ n) : 
    Gram.UD.IsCUDModOne (fun u : ℝ => Gram.Theorem4.gramPow n u)         
        \end{minted}
        Here, we again use auxiliary results from uniform distribution theory, which we prove in other files.
    \end{itemize}

 \textbf{Acknowledgments.} This project has received funding from the European Social Fund (project No 09.3.3-LMT-K-712-19-0030) under a grant agreement with the Research Council of Lithuania (LMTLT).

 \textbf{Disclosure statement.} The authors declare no conflict of interest.

\bibliographystyle{plain}
\bibliography{gramderivatives}
\end{document}